\documentclass[11pt]{article}

\usepackage{graphicx}
\usepackage{amsmath}
\usepackage{amsfonts}
\usepackage{amssymb}
\usepackage{amsthm}

\newtheorem{theorem}{Theorem}
\newtheorem{proposition}{Proposition}
\newtheorem{remark}{Remark}

\newtheorem{lemma}{Lemma}
\newtheorem{corollary}{Corollary}

\begin{document}

\title{A mathematical justification for metronomic chemotherapy in oncology} 

\author{  Luis A. Fern\'andez $^\dagger$
  \and
  Cecilia Pola $^\dagger$
  \and
  Judith S\'ainz-Pardo $^\star$
}

\newcommand{\Addresses}{{
  \bigskip
  \footnotesize

 $\dagger$ {Dep. Matem\'aticas, Estad\'{\i}stica y Computaci\'on, \par\nopagebreak Universidad de Cantabria,
Avda. de los Castros, s/n, 39005 Santander, Spain} \par

  \medskip

   $^\star$ {Instituto de F\'{\i}sica de Cantabria (UC-CSIC),
Avda. de los Castros, s/n, \par\nopagebreak 39005 Santander, Spain}\par\nopagebreak

 \medskip

\textit{E-mail address}, Luis A. Fern\'andez: \texttt{lafernandez@unican.es} \par\nopagebreak
\textit{E-mail address}, Cecilia Pola: \texttt{polac@unican.es}\par\nopagebreak
  \textit{E-mail address}, Judith S\'ainz-Pardo: \texttt{sainzpardo@ifca.unican.es}

}}

\maketitle

\Addresses

\begin{abstract}

We mathematically justify metronomic chemotherapy as the best strategy to apply most cytotoxic drugs in oncology for both curative and palliative approaches, assuming the classical pharmacokinetic model together with the Emax pharmacodynamic and the Norton-Simon hypothesis.

From the mathematical point of view, we will consider two mixed-integer nonlinear optimization problems, where the unknowns are the number of the doses and the quantity of each one, adjusting the administration times a posteriori.
\end{abstract}

\noindent{{\bf Mathematics Subject Classiﬁcation: }}{93C15, 92C50, 90C30}


\maketitle

\section*{Introduction}

Since the origins of chemotherapy in oncology, the standard and most widely used treatment has been the maximum tolerated dose (MTD) one: ``the higher the dose, the better",  whose rationale seems very intuitive. As everyone knows, the trade-off lies in the side effects due to drug toxicity, which not only diminishes quality of life for the patient (adding more illness to the existing one), but also make it necessary to impose rest periods
between cycles of therapy, which compromises a good resolution of the cancer treatment. For this reason, the design of more effective and less toxic chemotherapeutic regimens (by choosing the timing and fractionation of the doses) is a very active area of research in oncology.

A turning point in cancer chemotherapy treatments can be placed in the year $2000$, when a revolutionary alternative approach called ``metronomic chemotherapy" (MC) was proposed, see \cite{Scharovskyetal2009} and the references therein. According to National Cancer Institute \cite{NCI}, this is the ``treatment in which low doses of anticancer drugs are given on a continuous or frequent, regular schedule (such as daily or weekly), usually over a long time. Metronomic chemotherapy causes less severe side effects than standard chemotherapy. Giving low doses of chemotherapy may stop the growth of new blood vessels that tumors need to grow". In oncology, this last effect is referred to as ``anti-angiogenesis".

The historical controversy between efficacy in tumor killing and lack of toxicity has given rise to numerous studies and experimental trials comparing MTD and MC, the main conclusion of which can be summarized by saying that
``the body of experimental and clinical evidence, coupled with theoretical considerations [...] point to MC as a preferred course of action", see \cite{Karevaetal2015}.
Among the theoretical considerations, in addition to the anti-angiogenesis effect, we can mention that MC can activate anti-tumor immunity and minimize therapeutic resistance.

In this work, we will justify mathematically that MC is the best strategy in oncology for most cytotoxic drugs and both curative and palliative treatments, using simple well-established mathematical models and without taking these additional benefits into account. To that end, we will study two different optimization problems that arise when modeling  usual chemotherapy treatments for cancer:
to minimize the  tumor volume at a fixed final time (curative approach) or to minimize the total cumulative dose that a patient must take during a specific period of time for maintaining the tumor volume below a given threshold (palliative approach). We will assume a Gomperztian type tumor growth  and the Norton-Simon hypothesis to represent the growth-inhibitory influence due to the cytotoxic chemotherapy effect (see \cite{Norton-Simon1977, Simon-Norton2006}) with the classical pharmacokinetics and the Emax pharmacodynamic model. As far as we know, this is the first time that the superiority of MC is mathematically justified for  minimally parameterized models. Up to now, the efficiency of MC was related with heterogeneous tumor populations (sensitive/resistant cells) and palliative purposes, see \cite{Urszula2017} and the references therein.

From the mathematical point of view, the optimization problems under consideration are nonlinear and of mixed-integer type (where we will have the number of treatment doses as an integer variable, $N$, and the amount of each individual dose as a continuous variable, $d_i$). In both cases we will give explicit expressions of the approximated optimal solutions which will lead us to choose treatments corresponding to MC. A posteriori, the administration times can be adjusted (in fact, in several ways) completing the description of the treatment. We will illustrate the utility of our approach with numerical experiments concerning the treatment of some brain tumors  that exhibit Gompertzian growth (see \cite{Stensjoen-al2015}) using a cytotoxic drug called Temozolomide (TMZ). This drug is the subject of many clinical and biological recent studies (see \cite{BogdanskaPerezGarciaetal2017, FaivreBarbolosietal2013, Neyns_etal2010, Rossoetal2009} and their references).
The origin of this paper was an academic project \cite{TFG} submitted by the third author under the supervision of the other authors.

\section{Associated optimization problems}
\label{sec:AssociatedProblems}

The Gompertzian law has often been used to simulate the growth of some untreated tumors (see for instance \cite{Benzekry-etal2014}). In mathematical terms this can be described by the following ordinary differential equation (ODE):

\begin{equation}
L'(t) =\Psi(L(t)),
\label{GomperztPC}
\end{equation}
where
$$\Psi(L)= \xi L \log\left( \displaystyle\frac{\theta}{L} \right).$$

In the above equations, $L(t)$ represents the tumor volume (or tumor size)  at time $t$, $\xi$ its growth rate and $\theta$ the maximum size it can reach (also called carrying capacity for biological systems). In practice, the parameters $\xi$ and $\theta$ can also change with time (due to angiogenesis), but here we will focus on the simpler problem in which they are fixed.

As mentioned before, when using a cytotoxic drug to treat the tumor, we will consider the Norton-–Simon hypothesis to model its effect;
more precisely, we will assume that the growth-inhibiting effect due to the treatment is proportional to the growth rate of the untreated tumor
(see \cite{Norton-Simon1977, Simon-Norton2006}). Mathematically, this can be written as
\begin{equation}
 L'(t) =\Psi(L(t))(1-\rho(t)).
\label{NortonSimonPC}
\end{equation}
Here the term $\Psi(L(t))\rho(t)$ represents the growth-inhibitory influence due to the cytotoxic chemotherapy effect and it reflects both, the level of the therapy at time $t$, $\rho(t)$, and the tumor's sensitivity to therapy.

First let us present a  general result for the existence and uniqueness of solution for the Cauchy problem associated with (\ref{NortonSimonPC}) (see \cite{Fernandez-Pola2019} for the proof). As usual, we will denote by $W^{1,\infty}(0,T)$ the Sobolev space of all functions in $L^\infty(0,T)$ having first order weak derivative (in the distributional sense) also belonging to $L^\infty(0,T)$. It is well known that $W^{1,\infty}(0,T)$ can be identified with $C^{0,1}[0,T]$, the space of Lipschitz continuous functions in $[0,T]$, after a possible redefinition on a set of zero measure.

\begin{theorem}
Let us assume that $\xi$, $\theta$, $L_0$ and $T$ are given positive real numbers, $L_0 \in (0,\theta)$ and $\rho \in L^\infty(0,T)$. Then, there exists a unique solution $L \in W^{1,\infty}(0,T)$ of the following Cauchy problem
\begin{equation}
L'(t) =  \xi L(t) \log{\left(\frac{\theta}{L(t)}\right)} (1-\rho(t)), \ \ L(0) = L_0,
\label{GEG}
\end{equation}
given by
\begin{equation}
L(t) = \theta \exp{\left(\log{(L_0/\theta)}\exp{(-\xi \int_0^t (1-\rho(s))ds)} \right)}, \ \ \forall  t \in [0,T].
\label{E1}
\end{equation}
\label{T1}
\end{theorem}

Depending on the expression of $\rho(t)$, different pharmacodynamics (PD)  can be studied. In this work  we consider the classic Emax model:
\begin{equation}\label{fEMAX}
\rho(t) =  \frac{k_{1}c(t)}{k_{2}+c(t)},
\end{equation}
     where  $k_{1}$ and $k_{2} $ are fixed positive real numbers  and $c(t)$ denotes the concentration of the drug in the tumor at time $t$. Other choices have been  considered in the literature, as the Skipper model where $\rho (t) =  k_{1}c(t)$. The Emax model seems to be more appropriate from the clinical point of view because it  saturates for high concentration values.

Another crucial term is related with the pharmacokinetics (PK). We use the following Cauchy problem to describe the PK of the drug:
     \begin{equation}
\left\{ \begin{array}{l}
c'(t) = -\lambda c(t) + \displaystyle\sum^{N}_{i=1}\sigma d_{i}\delta (t-t_{i}),\\
             c(0) = 0 .\\
             \end{array}
\right.
\label{pCauchy}
\end{equation}
The coefficient $\lambda$ is the clearance rate and it is related to the half-life of the drug. The second term on the right hand of the equation depends on both the specific drug and the way it is administered. We assume that $N$ doses, $\{d_i\}_{i=1}^N$, will be used at $N$ dosage times, $\{t_i\}_{i=1}^N $, such that $ 0 \leq t_1 < \ldots < t_{N}$.
The coefficient $\sigma$ is determined on the basis of drug-, patient- and tumor-specific parameters. For instance, in the case of brain tumors  the drug loss during the transport to the brain has to be taking into account, see  \cite{BogdanskaPerezGarciaetal2017} and the references therein. In other cases, $\displaystyle \sigma = \frac{\alpha}{V_{D}\beta}$, where $\alpha$ is the patient's body surface area, $V_{D}$ is the volume of distribution of the drug and  $\beta$ is the patient weight in $kg$ if $V_{D}$ is given in $l/kg$. Finally,  $\delta(t-t_{i})$ is the Dirac delta distribution concentrated at $t_i$.

In principle, it seems reasonable to consider the administration times $\{t_i\}_{i=1}^N$ as variables of the problem, as well as $N$ and the doses $\{d_i\}_{i=1}^N$. However, this approach would lead us into serious difficulties because the criterion for deciding a priori the characteristics that determine the feasible protocols may be unknown even to the specialist. In practice, this would make the mathematical approach unfeasible. Instead, in a first step, we can consider the administration times as data, fixed according to some pattern. It is clear that this assumption will implicitly affect the solution obtained: for example, a daily regimen will inevitably lead to a solution with small doses, because MTD cannot be administered over a long period due to the serious side effects it produces. Apparently, it looks impossible to get out of this vicious circle, but we will see that this can be done for most cytotoxic drugs by following our approach and readjusting the administration times a posteriori, once $N$ and the doses $\{d_i\}_{i=1}^N$ have been determined.

Now, let us present our optimization problems: in this paper we are interested in the solutions of two optimization problems associated with the Cauchy problem (\ref{GEG}) using the Emax model (\ref{fEMAX}) and when the following data are given: the final time $T$, $ T > t_{N}  $ and the levels of effectiveness and the toxicity of each individual dose,  $d_{min}$ and  $d_{max}$, respectively. We assume $0< d_{min}< d_{max}$. To simplify some expressions we will use $t_{N+1}=T$. As we have explained before, the administration times $t_i$, $i = 1,\ldots,N$, are fixed a priori following some pattern, but the number $N$ of doses of the treatment is a variable to be determined.

Firstly, we consider a curative approach with the goal of minimizing the tumor volume at the final time $T$ with a fixed cumulative dose, $D$, that will be divided into $N$ smaller doses, $\{d_i\}_{i=1}^N$, to be determined. Hence we formulate the following problem with $N+1$ variables
    \begin{equation}
(P_{1}) \left\{ \begin{array}{lcc}
             \min \hspace{0.2cm} L(T),\\
        N \in \mathbb{N}, d=(d_1,\ldots,d_N) \in \mathbb{R}^{N},  \\
              \mbox{subject to   }  \displaystyle\sum_{i=1}^{N}d_i  = D, \\
              \hspace{1.9cm} d_{min} \leq d_i \leq d_{max}, \hspace{0.5cm}  i= 1, \ldots , N.
             \end{array}
   \right.
 \end{equation}
 Using (\ref{E1}) and the fact that $L_{0}<\theta$ (where $L_0$ is the initial tumor size),  we transform $(P_{1})$ into the following problem:
    \begin{equation}
\left\{
  \begin{array}{lcc}
             \max \hspace{0.2cm} \displaystyle\int_{0}^{T} \rho(s)ds,\\
             N \in \mathbb{N}, d=(d_1,\ldots,d_N) \in \mathbb{R}^{N},  \\
              \mbox{subject to   }  \displaystyle\sum_{i=1}^{N}d_i  = D, \\
              \hspace{1.9cm} d_{min} \leq d_i \leq d_{max}, \hspace{0.5cm} i= 1,\ldots,N.
             \end{array}
   \right.
   \label{tildeP1}
 \end{equation}

 Now, let's formulate the problem related to a palliative approach in which we want to minimize the total dose (determining $N$ and $d_i$, for $i=1, \ldots, N$) while maintaining the tumor volume at $T$ below a given threshold $L_{*}>0 $ that is not harmful to the patient, allowing him to have an acceptable quality of life. We consider the same bound constraints on the individual doses as in the previous optimization problem:
    \begin{equation}
    (P_{2}) \left\{ \begin{array}{lcc}
             \min \hspace{0.2cm} \displaystyle\sum_{i=1}^{N} d_i, \\
              N \in \mathbb{N}, d=(d_1,\ldots,d_N) \in \mathbb{R}^{N}, \\
              \mbox{subject to   }  L(T) \leq L_{*},
              \mbox{ } \\
              \hspace{1.8cm} d_{min} \leq d_i \leq d_{max}, \hspace{0.5cm} i = 1, \ldots, N .
             \end{array}
   \right.
   \end{equation}
 Using again  (\ref{E1}) it is easy to check the equivalence:
\begin{equation}
         L(T)\leq L_{*} \Longleftrightarrow
        \int_{0}^{T}\rho(s)ds \geq T + T_R, \mbox{with} \ \displaystyle T_R=\frac{1}{\xi}\log\left(\frac{\log\left(L_*/\theta\right)}
{\log\left(L_0/\theta\right)}\right).
        \label{cota}
\end{equation}
 Therefore,  we can reformulate the problem $(P_{2})$ as follows:
\begin{equation}
\left\{
\begin{array}{lcc}
             \min \hspace{0.2cm} \displaystyle\sum_{i=1}^{N}d_{i}, \\
             N \in \mathbb{N}, d=(d_1,\ldots,d_N) \in \mathbb{R}^{N}, \\
              \mbox{subject to   }  \displaystyle \int_{0}^{T} \rho(s)ds \geq T + T_R, \\
              \hspace{2cm} d_{min} \leq d_i \leq d_{max}, \hspace{0.5cm}  i = 1, \ldots, N .
             \end{array}
  \right.
   \label{tildeP2}
\end{equation}

\section{Optimal treatments}

In view of our optimization problems it is convenient to have an explicit formulation for the concentration of the drug.
\begin{proposition}
The solution of the Cauchy problem \eqref{pCauchy} is given by
\begin{equation}
        c(t)=\left\{ \begin{array}{ll}
            0 ,& \hspace{0.3cm} \mbox{if}\ \  t \in [0,t_{1}),\\
            \sigma e^{-\lambda t}e^{\lambda t_{1}}d_{1} , & \hspace{0.3cm} \mbox{if}\ \ t \in [t_{1},t_{2}),\\
            \sigma e^{-\lambda t}(e^{\lambda t_{1}}d_{1} + e^{\lambda t_{2}}d_{2}) , & \hspace{0.3cm} \mbox{if}\ \ t \in [t_{2},t_{3}),\\
            \vdots & \hspace{0.3cm} \vdots\\
            \sigma e^{-\lambda t}(e^{\lambda t_{1}}d_{1} + \ldots +e^{\lambda t_{N-1}}d_{N-1}) , & \hspace{0.3cm} \mbox{if}\ \ t \in [t_{N-1},t_{N}),\\
            \sigma e^{-\lambda t}(e^{\lambda t_{1}}d_{1} + \ldots +e^{\lambda t_{N}}d_{N}) , & \hspace{0.3cm} \mbox{if}\ \ t \in [t_{N},T).\\
        \end{array}
    \right.
    \label{c(t)}
    \end{equation}
\end{proposition}

\begin{proof}
Applying Laplace transform to problem (\ref{pCauchy}) we get
$$ s{\cal L}(c)(s) = -\lambda {\cal L}(c)(s) + \sum^{N}_{i=1}\sigma d_{i}e^{-s t_i},$$
and consequently,
$$ {\cal L}(c)(s) = \sum^{N}_{i=1}\sigma d_{i}\frac{e^{-s t_i}}{s+\lambda} = \sum^{N}_{i=1}\sigma d_{i} e^{-s t_i}{\cal L}(e^{-\lambda t})(s)= \sum^{N}_{i=1}\sigma d_{i}{\cal L}(g_i)(s),$$
where
$$g_i(t) = \left\{ \begin{array}{ll}
            0 ,& \hspace{0.3cm} \mbox{if}\ \  0 \leq t \leq t_i,\\
            e^{-\lambda (t-t_i)} , & \hspace{0.3cm} \mbox{if}\ \ t \geq t_i.
            \end{array}
\right.$$
This implies that
$$ c(t) = \sum^{N}_{i=1}\sigma d_{i}g_i(t),$$
which is equivalent to (\ref{c(t)}).
\end{proof}

We continue by obtaining the expression for the integral term appearing in \eqref{tildeP1} and \eqref{tildeP2}. In this respect, we present the following result:

\begin{lemma}
\label{intrhoEmax}
Let us assume that $c(t)$ is given by \eqref{c(t)} and $\rho (t)$ by \eqref{fEMAX}. Then, it is verified that
\begin{equation}
\label{intrhoEm}
 \int_0^T \rho(s) \; ds =  \frac{k_{1} }{\lambda}
 \log \left( \frac{(d_1+\tilde{k}_{2} )  {\displaystyle\prod_{i=2}^{N}}
( {\displaystyle\sum_{j=1}^{i-1}} d_{j}e^{\lambda(t_{j}-t_{i})} + d_{i} + \tilde{k}_{2}) }{ {\displaystyle\prod_{i=1}^{N}} ({\displaystyle\sum_{j=1}^{i} } d_{j}e^{\lambda(t_{j}-t_{i+1})}+\tilde{k}_{2})}
    \right) ,
\end{equation}
where $\tilde{k}_{2}=k_{2}/\sigma$ and $t_{N+1}=T$.
\end{lemma}

\begin{proof}
 Using (\ref{c(t)}) and the definition of $\rho$, we derive \eqref{intrhoEm} from the following equalities:
\[    \int_{t_1}^{t_2} \displaystyle\frac{c(s)}{k_{2}+c(s)}ds    = \frac{1}{\lambda}\log\left(\frac{d_{1}+\tilde{k}_{2}}{d_{1}e^{\lambda(t_{1}-t_{2})}+\tilde{k}_{2}}\right),\]

\[ \int_{t_{i}}^{t_{i+1}}\displaystyle\frac{c(s)}{k_{2}+c(s)}ds = \frac{1}{\lambda}\log\left(\frac{d_{1}e^{\lambda(t_{1}-t_{i})}+\cdots+d_{i-1}e^{\lambda(t_{i-1}-t_{i})}+d_{i}+\tilde{k}_{2}}{d_{1}e^{\lambda(t_{1}-t_{i+1})}+\cdots+d_{i}e^{\lambda(t_{i}-t_{i+1})}+\tilde{k}_{2}}
\right) ,   \]
for $i=2, \ldots, N-1$ and taking into account that the integral corresponding to the interval $[t_{N},T]$ is similar to the previous expression changing $t_{i}$ by $t_{N}$ and  $t_{i+1}$ by $T$.
\end{proof}

\subsection{Curative approach}

 We start by considering the problem $(P_{1})$ defined in Section \ref{sec:AssociatedProblems} (in which we want to minimize the tumor volume at time $T$  with a fixed cumulative dose $D$). In this case, using \eqref{tildeP1}, \eqref{intrhoEm}, the  positivity  of $k_{1}/ \lambda $ and the monotonicity of the logarithm function, we reformulate the problem $(P_{1})$ as follows:

\begin{equation}
 \left\{ \begin{array}{lcc}
             Max \hspace{0.2cm} f_{1}(N,d) = {\displaystyle\frac{(d_1+\tilde{k}_2 )  {\displaystyle\prod_{i=2}^{N}}
( {\displaystyle\sum_{j=1}^{i-1}} d_{j}e^{\lambda(t_{j}-t_{i})} + d_{i} + \tilde{k}_{2}) }{ {\displaystyle\prod_{i=1}^{N}} ({\displaystyle\sum_{j=1}^{i} } d_{j}e^{\lambda(t_{j}-t_{i+1})}+\tilde{k}_{2})}}, \\
        N \in \mathbb{N}, d=(d_1,\ldots,d_N) \in \mathbb{R}^{N},  \\
              \mbox{subject to   }  \displaystyle\sum_{i=1}^{N}d_i  = D,\\
              \hspace{1.9cm} d_{min} \leq d_i \leq d_{max}, \hspace{0.5cm} i = 1, \hdots, N.
             \end{array}
   \right.
  \label{P1E}
 \end{equation}

Note that it is a mixed-integer nonlinear optimization problem with an integer variable, $N$, and $N$ continuous variables, $d_i$.

In order to simplify the objective function $f_1$, along the rest of the work, we will assume the {\bf main hypothesis} that we are considering cytotoxic drugs verifying
\begin{equation}
d_{max}e^{-\lambda s} \ll \tilde{k}_2,
\label{MH}
\end{equation}
where $s$ denotes the minimum of the elapsed time periods; this is
$$s=\min\{ t_{i+1}-t_i: i=1, \ldots, N\}.$$
Let us stress that this condition holds
(at least) when $ \lambda$ is sufficiently large (or equivalently, the half-life of the drug is sufficiently small) and that most of the usual cytotoxic drugs satisfy this condition, even when $s$ is small.

Assuming (\ref{MH}), the following approximation for the objective function holds:
\begin{equation}
\label{aproxbiglambda}
  f_{1}(N,d) \approx \displaystyle\prod_{i=1}^{N}  \left( \frac{d_{i}}{\tilde{k}_{2}}+1\right),
  \end{equation}
that leads to the mixed-integer nonlinear programming problem:
\begin{equation}
(\hat{P}_{1}) \left\{ \begin{array}{lcc}
             Max \hspace{0.2cm} \hat{f}_{1}(N,d) = \displaystyle\prod_{i=1}^{N}\left(\frac{d_{i}}{\tilde{k}_{2}}+1\right), \\
              N \in \mathbb{N}, d=(d_1,\ldots,d_N) \in \mathbb{R}^{N},  \\
              \mbox{subject to   }  \displaystyle\sum_{i=1}^{N}d_i  = D,\\
              \hspace{1.9cm} d_{min} \leq d_i \leq d_{max}, \hspace{0.5cm} i = 1, \hdots, N.
             \end{array}
   \right.
   \label{hatP1E}
 \end{equation}

At this point, it is very important to underline that the administration times ${\{t_i\}}_{i=1}^N$ do not appear in the formulation of $(\hat{P}_{1})$.
This is a crucial property that will allow us to break the vicious circle mentioned in the previous section:
for most cytotoxic drugs we will first determine $N$ and the doses ${\{d_i\}}_{i=1}^N$ as the solution of $(\hat{P}_{1})$
and then assign appropriate administration times, checking that (\ref{MH}) holds.

Using the bound constraints and the total dose value, in order to have a non-empty set of feasible points, we must assume
\begin{equation}
[D/ d_{max}, D/ d_{min}] \cap \mathbb{N} \neq \emptyset.
\label{H1}
\end{equation}

In particular, hypothesis (\ref{H1}) implies that only a finite number of values for variable $N$ have to be taken into account to solve the optimization problem. Namely,
\begin{equation}
\label{condNb}
 N \in \{\lceil D/ d_{max} \rceil,\ldots, \lfloor D/ d_{min} \rfloor\},
\end{equation}
where $\lfloor x \rfloor$ denotes the greatest integer less than or equal to $x$ and $\lceil x \rceil$ the least integer greater than or equal to $x$.

Now, for each fixed feasible value of $N$, we can formulate the following continuous nonlinear programming problem:
\begin{equation}
\label{hatP1EN}
(\hat{P}_{1}^N) \left\{ \begin{array}{lcc}
             Max \hspace{0.2cm} \hat{f}_{1}^N (d) =  \displaystyle\prod_{i=1}^{N}\left(\frac{d_{i}}{\tilde{k}_{2}}+1\right),\\
             d=(d_1,\ldots,d_N) \in \mathbb{R}^{N},  \\
              \mbox{subject to   }  \displaystyle\sum_{i=1}^{N}d_i = D,\\
              \hspace{1.9cm} d_{min} \leq d_i \leq d_{max}, \hspace{0.5cm} i = 1, \hdots, N.\\
             \end{array}
\right.
\end{equation}

It is not difficult to prove that equal-dosage is the optimal solution of the previous problem:

\begin{theorem}
Let us assume that $d_{min}>0$ and (\ref{H1}). Then, for each $N$  verifying \eqref{condNb},
the problem $(\hat{P}_{1}^N)$ has a unique optimal solution $\hat{d}$ given by
$$ \hat{d}_i= D/N,   \ \ \mbox{for} \ \ i=1, \ldots, N.$$
\label{teoremaMaxP1ENbigLambda}
\end{theorem}
\begin{proof}
For any feasible point $d$, using the relation between the geometric and arithmetic means, we get
$$\sqrt[N]{\prod_{i=1}^{N}\left(\frac{d_{i}}{\tilde{k}_{2}}+1\right)} \leq \frac{1}{N}{\displaystyle \sum_{i=1}^{N}\left(\frac{d_{i}}{\tilde{k}_{2}}+1\right) = \frac{D}{N\tilde{k}_2}+1 },$$
and therefore,   $\hat{f}_{1}^N (d)\leq \left(\frac{D}{N\tilde{k}_{2}}+1\right)^{N}$. By the hypotheses, $\hat{d} = (D/N,\ldots,D/N)$ is a feasible point for $(\hat{P}_{1}^N)$ such that the
corresponding objective function value reaches the above upper bound, so the conclusion of the theorem follows straightforwardly.
\end{proof}

As a consequence,  the optimal solution of the approximate mixed-integer optimization problem $(\hat{P}_{1})$ leads to the longest feasible treatment:
\begin{corollary}
Let us assume that $d_{min}>0$ and (\ref{H1}). Then the optimal solution of $(\hat{P}_{1})$ is given by  $(\hat{N},\hat{d})$, with  $\hat{N}= \lfloor D/ d_{min} \rfloor$ and
$\hat{d}_i=D/\hat{N},$ for $i=1, \ldots,\hat{N}.$
\label{coroteoremaMaxP1EN}
\end{corollary}

\begin{proof}
This can be deduced combining Theorem \ref{teoremaMaxP1ENbigLambda} with the fact that the auxiliary function
$$\varphi_1(x)=\left(\frac{1}{x}+1\right)^{x}$$
is strictly increasing in $(0,+\infty)$ and noticing that
$$\hat{f}_1(N,\hat{d}) = \left(\frac{D}{N\tilde{k}_{2}}+1\right)^N = \left(\varphi_1\left(\frac{N\tilde{k}_{2}}{D}\right)\right)^{\frac{D}{\tilde{k}_{2}}}.$$
\end{proof}

Furthermore, we can   study  the dependence of the objective function  optimal value with respect to $d_{min}$. Using previous expressions, we derive
$$   \hat{f}_{1}(\hat{N},\hat{d}) = \left(\frac{D}{\hat{N}\tilde{k}_{2}}+1\right)^{\hat{N}} \approx \left(\frac{d_{min}}{\tilde{k}_{2}}+1\right)^\frac{D}{d_{min}} =
\left(\varphi_1\left(\frac{\tilde{k}_{2}}{d_{min}}\right)\right)^{\frac{D}{\tilde{k}_{2}}}.$$
On the one hand, this implies that to increase $\hat{f}_{1}(\hat{N},\hat{d})$ (and therefore to decrease the tumor size at the fixed final time) is interesting to take $d_{min}$ as small as possible, in line with the principles of MC.
But, on the other hand, from the clinical point of view, it is known that for any drug there is a minimum effective dose (MED) (the smallest dose with a discernible useful effect) that must be taken into account. In general, the estimation of the MED (and even its definition) remains controversial and this carries over to the determination of the optimal $d_{min}$ in the practice.

\begin{remark}
\begin{itemize}
\item [i)] Choosing the Gompertzian law for the untreated tumor growth is not a crucial aspect for deriving our results.
Other typical ODEs with sigmoid-type solutions may be reasonable choices, such as (\ref{GomperztPC}) with
$$\Psi(L)= \frac{\xi}{\gamma} L\left(1-\displaystyle{\left(\frac{L}{\theta}\right)^\gamma}\right),$$
with $\gamma >0$.  For $\gamma = 1$, we recover the well-known logistic ODE, meanwhile for small $\gamma >0$ ,
this is an approximation of the Gompertz ODE, see \cite{Benzekry-etal2014}. All these options lead to the same results.

\item [ii)] Instead, it can be shown that the Norton-Simon hypothesis is crucial for our argument.
There are other classical hypotheses such as the log-kill one. It states that a given dose of chemotherapy
kills the same fraction of tumor cells regardless of the size of the tumor at the time of treatment, which leads
to the following problem
\begin{equation}
L'(t) =  \xi L(t) \log{\left(\frac{\theta}{L(t)}\right)} -\rho(t)L(t), \ \ L(0) = L_0,
\end{equation}
instead of (\ref{GEG}). Although the log-kill hypothesis continues to appear frequently in the literature,
it was shown in \cite{Norton-Simon1977} to be inconsistent with some clinical experiences.

\item [iii)] We have studied a related problem for continuous (non-discrete) drug infusion in \cite{Fernandez-Pola2019} (labeled $(OP_2)$ with $G = G_2$).
There, it was proved the appearance of a constant maintenance infusion rate (during a fairly long time interval) in the expression of the optimal control.
This can be seen as a piecewise continuous version of our previous Corollary \ref{coroteoremaMaxP1EN}.

\end{itemize}
\end{remark}

From a theoretical point of view, we have seen that the optimal curative treatment (i.e. aiming to minimize the tumor size at a given time with a fixed cumulative dose using drugs for which \eqref{aproxbiglambda} holds) is the longest feasible treatment with equal individual doses. This is the case for drugs with a short half-life such as Temozolomide (TMZ), an oral alkylating agent with antitumor activity in high and low grade gliomas (HGG and LGG, respectively). In the monotherapy phase for adult patients with newly-diagnosed glioblastoma multiforme, TMZ is usually administered consecutively for $5$ days of every $28$ days (this is denoted by  $5/28d$) for a maximum of $6$ cycles  with individual doses of $150$ $mg/m^2$ in first cycle  and $200$ $mg/m^2$ in the others.  So, taking into account this usual treatment (UT),  for our numerical experiments we have  fixed $T=210$ days for the final time and $D=5{,}750$ $mg/m^2$ for the cumulative dose.

 The parameter values used for our numerical computations are showed in Table \ref{ParameterValues}.
 Following the reference cited in \cite{BogdanskaPerezGarciaetal2017} we know that TMZ  has a mean elimination half-life  of $1.8$ hours (i.e. $t_{1/2}= 0.075$ days).
 The clearance rate $\lambda$ is estimated using that $\lambda= log(2)/t_{1/2} = 9.242$ $days^{-1}$. The growth rate $\xi$ depends on the type of tumor; the value $\xi = 5.51e$$-3$ could correspond to an HGG with  a volume-doubling time about $126$ days (see \cite{Stensjoen-al2015} for growth dynamics of glioblastomas). Following  \cite{BogdanskaPerezGarciaetal2017}, we have used $\sigma=1.6(2.5e$$-3) \approx 4e$$-3$ that corresponds to a woman  with a body surface  of $1.6 \; m^2$ and it takes into account the fraction of TMZ getting to her brain interstitium. On the other hand, it is well known that $k_1$ represents the maximum effect of the drug on the tumor and $k_2$ is the effective concentration, i.e  concentration producing $50\%$ of the maximum effect.
We have chosen $k_2 = 0.36 \ mg/l$, a value that is consistent with predicted peak concentrations of TMZ from  \cite{BogdanskaPerezGarciaetal2017, Rossoetal2009}. Moreover, we have taken $k_1=60$ for illustrative purposes.
Let us note that the parameter $k_1$ does not play any role either in the
optimization process of this subsection nor in the main hypothesis (\ref{MH}), but it does influence the evaluation of the tumor size. Let us stress that (\ref{MH}) holds in this case, because
$$d_{max}e^{-\lambda s} \approx 0.01938 \ll \tilde{k}_2 = 90,$$
with $s=1$.
Finally, let us underline that the value of $\theta$ is unnecessary here, because the initial datum and the results are given rescaled in the form $L_0/\theta$ and $L(t)/\theta$, resp. (see Table \ref{ParameterValues} and Figure~\ref{fig:1E}).

\begin{table}
\caption{Parameter values used for the numerical experiments.}
\label{ParameterValues}
\begin{tabular}{l|c}
\hline\noalign{\smallskip}
 Parameters &    Temozolomide (TMZ)  \\
\noalign{\smallskip}\hline\noalign{\smallskip}
$\lambda$  ($days^{-1}$)  &  9.242   \\
$\xi$   ($days^{-1}$)    &  5.51e$-3$   \\
$k_{1}$                  & 60       \\
$k_{2}$  ($mg/l$)        &  3.6e$-1$        \\
$\sigma$  ($m^2/l$)      & 4e$-3$     \\
$d_{max}$  ($mg/m^2$)    & 200          \\
$T$ ($days$)             &  210          \\
$L_0/\theta$             & 2.5e$-1$      \\
$t_1$                    &  0             \\
\hline
\end{tabular}
\end{table}

 In  Table~\ref{1ETMZ} we present some numerical results to illustrate and compare the optimal treatments corresponding to $(P_{1})$ without using the approximation \eqref{aproxbiglambda} (i.e. (\ref{P1E})) with those associated to $(\hat{P}_{1})$ (given by  Theorem~\ref{teoremaMaxP1ENbigLambda} and Corollary~\ref{coroteoremaMaxP1EN}) in the case of individual doses between $d_{min} = 100$ $mg/m^2$ and $d_{max} = 200$ $mg/m^2$. To solve the problems associated with $(P_{1})$ when the value for $N$ is fixed, we have used FMINCON (a nonlinear programming solver provided in MATLAB's Optimization Toolbox) with the SQP algorithm and the value $1.e$$-8$ for the optimality tolerance.
 Here, we present results only for the feasible  values of $N$ taking into account the length of the treatment interval and the number of dosage times for the 5/28d schedule: $N \in [29,40]$.
 This number $N$ appears in the first column; then we present two columns with the individual doses of the treatment described in Theorem \ref{teoremaMaxP1ENbigLambda} and the ratio of the final tumor volume to the initial  volume, $L(T)/L_0$, for an initial medium size tumor ($25\%$ of carrying capacity); the next three columns report results given by FMINCON: the  minimum and maximum values reached by the individual doses of each treatment, $\bar{d}_{min}$ and  $\bar{d}_{max}$, and its corresponding ratio $L(T)/L_0$. Let us note at this point that, for each value of $N$,  the individual doses obtained with FMINCON are almost identical and the optimal solution is reached administering $40$ doses  (the longest 5/28d treatment, as Theorem~\ref{teoremaMaxP1ENbigLambda} and Corollary~\ref{coroteoremaMaxP1EN} state, using the approximation \eqref{aproxbiglambda}). The last line of the table is associated to the usual treatment (UT) described above using $N=30$ doses. It is noteworthy that with it the final tumor volume is $91\%$ of the initial volume while our optimal treatment provides $73\%$. At the top of Figure~\ref{fig:1E} we see the evolution of the tumor volume, $L(t)/\theta$, for a selection of treatments of Table~\ref{1ETMZ}.

\begin{table}
\caption{Results  with $D=5{,}750$ $mg/m^2$ and $d_{min} = 100$ $mg/m^2$ with 5/28d schedule.}
\label{1ETMZ}
\begin{tabular}{l|cc||ccc}
\hline\noalign{\smallskip}
 & \multicolumn{2}{|c||}{$(\hat{P}_{1})$} &  \multicolumn{3}{|c}{$(P_{1})$} \\
   \noalign{\smallskip}
\hline\noalign{\smallskip}
 $N$  & $D/N$ &  $L(T)/L_0$  &   $\bar{d}_{min}$ & $\bar{d}_{max}$ &    $L(T)/L_0$  \\
\hline
  $  29 $ & $ 198.28 $ & $ 0.93 $ & $ 198.26 $ & $ 198.32 $ & $ 0.93 $
\\
 $  30 $ & $ 191.67 $ & $ 0.91 $ & $ 191.65 $ & $ 191.71 $ & $ 0.91 $
\\
$  31 $ & $ 185.48 $ & $ 0.89 $ & $ 185.47 $ & $ 185.54 $ & $ 0.89 $
\\
$  32 $ & $ 179.69 $ & $ 0.87 $ & $ 179.67 $ & $ 179.72 $ & $ 0.87 $
\\
 $  33 $ & $ 174.24 $ & $ 0.85 $ & $ 174.23 $ & $ 174.28 $ & $ 0.85 $
\\
$  34 $ & $ 169.12 $ & $ 0.83 $ & $ 169.10 $ & $ 169.15 $ & $ 0.83 $
\\
$  35 $ & $ 164.29 $ & $ 0.81 $ & $ 164.27 $ & $ 164.32 $ & $ 0.81 $
\\
 $  36 $ & $ 159.72 $ & $ 0.79 $ & $ 159.71 $ & $ 159.77 $ & $ 0.79 $
 \\
 $  37 $ & $ 155.41 $ & $ 0.77 $ & $ 155.39 $ & $ 155.43 $ & $ 0.77 $
 \\
 $  38 $ & $ 151.32 $ & $ 0.76 $ & $ 151.30 $ & $ 151.34 $ & $ 0.76 $
\\
 $  39 $ & $ 147.44 $ & $ 0.74 $ & $ 147.42 $ & $ 147.46 $ & $ 0.74 $
\\
 $  \mathbf{40} $ & $ 143.75$ & $ \mathbf{0.73} $ & $ \mathbf{143.74} $ & $ \mathbf{143.78} $ & $ \mathbf{0.73} $ \\ \hline
 30 (UT) &    &  & $ 150 $ &   $ 200 $ &  $ 0.91 $
\\
\noalign{\smallskip}\hline
\end{tabular}
\end{table}

\begin{figure}[htp]
\begin{center}
  \includegraphics[width=4.3in]{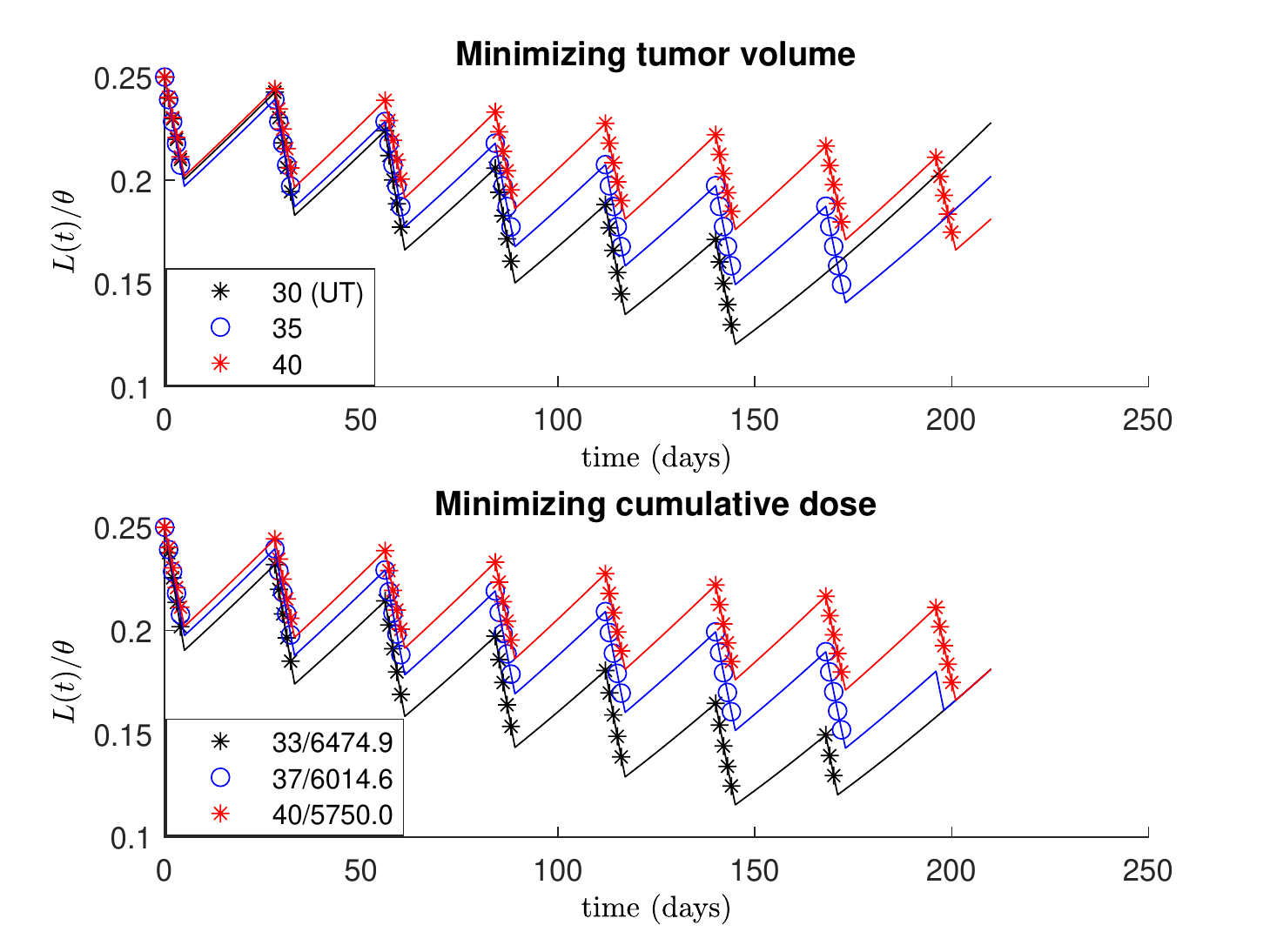}\\
   \caption{Optimization results for TMZ (see Tables~\ref{1ETMZ} and \ref{2ETMZ}).}
   \label{fig:1E}
  \end{center}
\end{figure}

 The most relevant results of this section are given in  Table~\ref{1EdminTMZ} and at the top of Figure~\ref{fig:1Edmin}, where we illustrate the influence of the level $d_{min}$ and compare several optimal treatments corresponding to $(P_{1})$ requiring the same cumulative dose ($D=5{,}750$  $mg/m^2$). The rest of the parameters are given in Table \ref{ParameterValues}. As we have mentioned previously, first we have fixed $d_{min}$ and solved $(\hat{P}_{1})$; once we have determined the corresponding solution ($\hat{N}$ and $\{\hat{d}_i\}_{i=1}^{\hat{N}}$) we have selected the most appropriate schedules for TMZ used in some studies, see  \cite{Neyns_etal2010}: the 7/14d with doses of $150$ $mg/m^2$, the 21/28d with doses of $100$ $mg/m^2$ and $75$ $mg/m^2$ and the daily regimen with $50$ $mg/m^2$ doses. The results of the penultimate column of Table~\ref{1EdminTMZ} have been calculated with the times given by these schedules. Let us underline that the specific values of the administration times do not affect the optimal solution and the only relevant factor is the maximum number of doses that can be administered with the chosen $d_{min}$. We observe that the greatest tumor shrinkage is achieved with the longest treatment which also has the lowest doses, in line with MC.  Moreover, three of the treatments of  Table~\ref{1EdminTMZ}  are more effective than the optimal treatment of  Table~\ref{1ETMZ}, all of them improving the usual treatment UT.

 As predicted by the Norton-Simon hypothesis \cite{Simon-Norton2006}, it is apparent that the optimal solution is more ``dose-dense" than other feasible protocols, see the top of Figure \ref{fig:1Edmin} (the period between cycles is shorter). Nevertheless, concerning the ``dose-intensity" (defined as the total dose administered during a treatment, divided by its duration) it can be seen in the last column of  Table \ref{1EdminTMZ} that the optimal protocol has the lowest one.

\begin{table}
\caption{Numerical results for $(\hat{P}_{1})$ with $D=5{,}750$ $mg/m^2$ and varying $d_{min}$.}
\label{1EdminTMZ}
\begin{tabular}{l|cc|lcc}
\hline\noalign{\smallskip}
$d_{min}$  &  $\hat{N}$  & $D/\hat{N}$ & Schedule &  $L(T)/L_0$ & Dose intensity \\
\noalign{\smallskip}\hline\noalign{\smallskip}
$150$   &  $  38 $ & $ 151.32 $ &  7/14d & $ 0.76 $ & 78.77 \\
$100$   &  $  57 $ & $ 100.88 $ & 21/28d & $ 0.53 $ & 80.99 \\
$75 $   &  $  76 $ & $  75.66 $ & 21/28d & $ 0.41 $ & 59.28 \\
$50 $   &  $ 115 $ & $  50.00 $ & 28/28d& $ 0.27 $  & 50.00  \\
\hline
\end{tabular}
\end{table}

\begin{figure}[htp]
\begin{center}
 \includegraphics[width=4.3in]{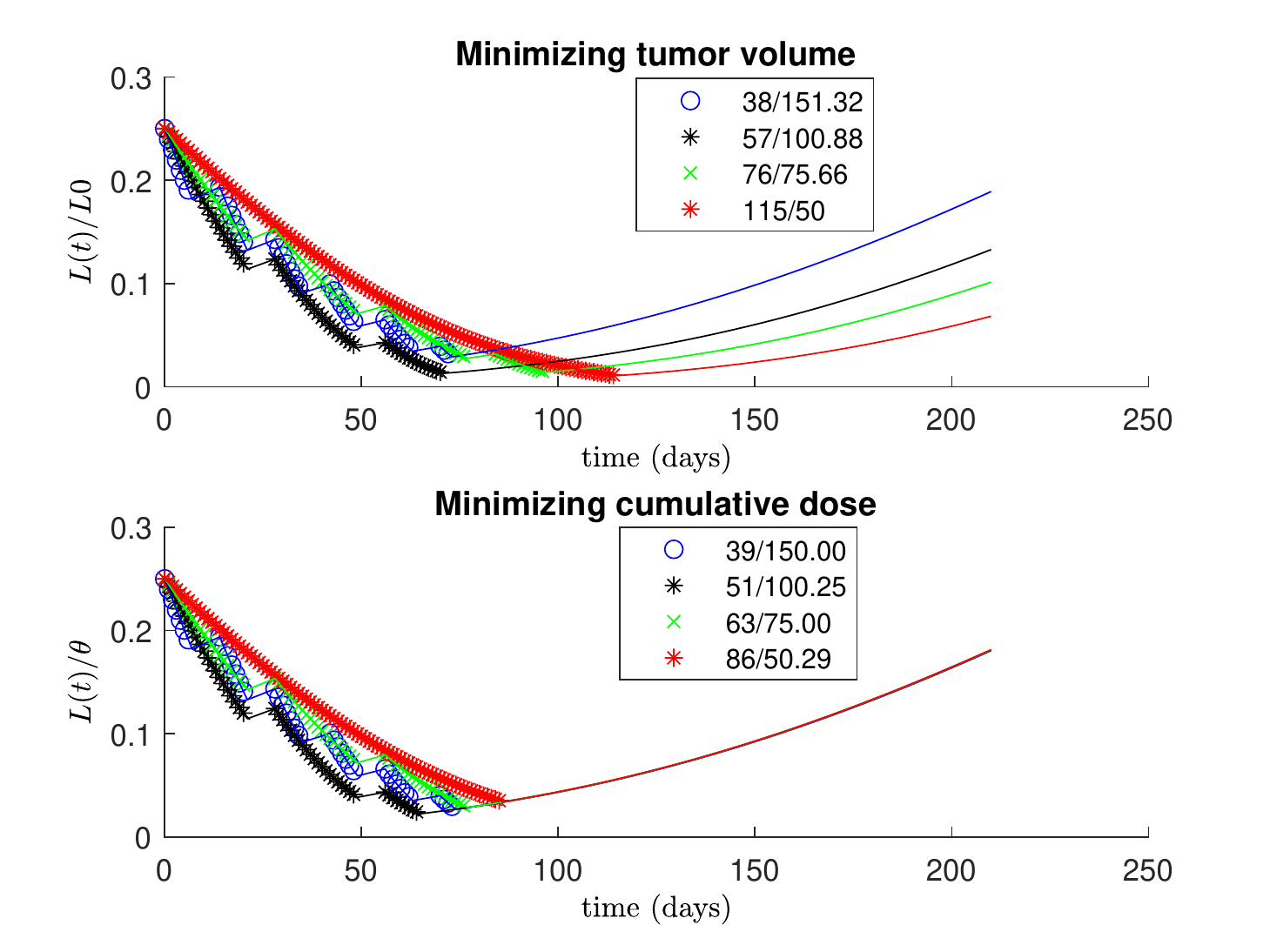}\\
  \caption{Optimization results with TMZ (see Tables~\ref{1EdminTMZ} and \ref{2EdminTMZ}).}
   \label{fig:1Edmin}
  \end{center}
\end{figure}

\subsection{Palliative approach}

In this section we are interested in the optimization problem $(P_{2})$ i.e. we want to minimize the total administered dose with a constraint on the tumor size at the final time: $L(T) \leq L_{*}$, where $L_{*}$ is a given level beforehand. Using (\ref{cota}) and Lemma~\ref{intrhoEmax}, that constraint is equivalent to the following inequality:
$$ f_{1}(N,d) \geq e^{\tilde{T}_R},  $$
being $f_{1}$ the objective function of problem \eqref{P1E}  and
\begin{equation}
 \tilde{T}_R \stackrel{def}{=} \frac{\lambda}{k_{1}}(T+T_R).
\label{cotaEmax}
\end{equation}

Let us remember that  $T_R$  depends on the data $\xi$, $\theta$, $L_{0}$ and $L_{*}$ (see \eqref{cota}).
Therefore, using \eqref{tildeP2}, the optimization problem $(P_{2})$  can be reformulated as follows:

\begin{equation}
 \left\{ \begin{array}{lcc}
             \min \hspace{0.2cm} f_2(N,d) = \displaystyle\sum_{i=1}^{N} d_i, \\
        N \in \mathbb{N}, d=(d_1,\ldots,d_N) \in \mathbb{R}^{N},  \\
              \mbox{subject to   }  f_1(N,d) \geq e^{\tilde{T}_R} ,\\
              \hspace{1.9cm} d_{min} \leq d_i \leq d_{max}, \hspace{0.5cm} i = 1, \hdots, N,
             \end{array}
   \right.
 \end{equation}
with $d_{min} > 0$ and $\tilde{T}_R \in \mathbb{R}$ defined  in (\ref{cotaEmax}). Using the approximation \eqref{aproxbiglambda} in the nonlinear constraint we arrive to the following nonlinear programming problem
\begin{equation}
(\hat{P}_{2}) \left\{ \begin{array}{lcc}
             \min \hspace{0.2cm} f_2(N,d) = \displaystyle\sum_{i=1}^{N} d_i,  \\
        N \in \mathbb{N}, d=(d_1,\ldots,d_N)  \in \mathbb{R}^{N},  \\
              \mbox{subject to   }  \hat{f}_{1} (N,d) \geq e^{\tilde{T}_R},\\
              \hspace{1.9cm} d_{min} \leq d_i \leq d_{max}, \hspace{0.5cm} i = 1, \hdots, N,
             \end{array}
   \right.
 \end{equation}
with $\hat{f}_{1}$ defined in (\ref{hatP1E}). As one can see, $(\hat{P}_{2})$ is also a mixed-integer nonlinear  optimization problem with an integer variable $N$ and $N$ continuous variables, $d_i$.

In this case, let us remark that the set of feasible values for $N$ is infinite: using the general inequality constraint and the upper bound constraints, the following inequality  should be satisfied
 $$N \geq \lceil \tilde{T}_R/log(d_{max}/\tilde{k}_2 +1) \rceil \stackrel{def}{=} N_{min}.$$
 In fact, only a finite subset has practical interest, i.e. those values that also verify
 $$N \leq  N_{max} \stackrel{def}{=} \lceil \tilde{T}_R/log(d_{min}/\tilde{k}_2 +1) \rceil,$$
  because at least the pair
$(N_{max},\underbrace{d_{min}, \ldots,d_{min}}_{N_{max}})$ is feasible and any other feasible pair with larger value of $N$ will produce a greater value of the objective function $f_2$.
Therefore, to have a non-empty set of feasible points we will assume

\begin{equation}
[N_{min}, N_{max}] \cap \mathbb{N} \neq \emptyset.
\label{H2}
\end{equation}

Now, as in the previous subsection, we will focus on the associated continuous problems generated when the value of the integer variable $N$ is fixed:
\begin{equation}
(\hat{P}_{2}^N) \left\{ \begin{array}{lcc}
             \min \hspace{0.2cm} f_2^N (d) = \displaystyle\sum_{i=1}^{N} d_i, \\
              d=(d_1,\ldots,d_N)  \in \mathbb{R}^{N},\\
              \mbox{subject to   } \  \hat{f}_{1}^N (d) \geq e^{\tilde{T}_R},\\
              \hspace{1.9cm} d_{min} \leq d_i \leq d_{max}, \hspace{0.5cm} i = 1,\hdots,N,\\
             \end{array}
   \right.
  \end{equation}
with the function $\hat{f}_{1}^N$ defined in \eqref{hatP1EN}. In the following result we show that the optimal treatment for $(\hat{P}_{2}^N)$ has also equal doses and moreover, we obtain the explicit formula to determine them:

\begin{theorem}
Let us assume $d_{min}>0$ and
\begin{equation}
N \in [N_{min}, N_{max}-1].
\label{H2b}
\end{equation}
Then, the optimization problem $(\hat{P}_{2}^N)$ has a unique solution $\hat{d}=(\hat{d}_1,\ldots,\hat{d}_N)$ given by
\begin{equation}
 \hat{d}_i= \tilde{k}_{2}(e^{\tilde{T}_R/N}-1), \ \ \ \mbox{for} \ \ i=1, \ldots, N.
\end{equation}
\label{teoremaMaxP2ENbigLambda}
\end{theorem}

\begin{proof} First of all, it is immediate to check that the proposed dose vector is feasible for the problem $(\hat{P}_{2}^N)$, if  $N$  verifies \eqref{H2b}.
On the other hand, using once more the relation between the geometric and arithmetic means, the doses of any feasible treatment for $(\hat{P}_{2}^N)$ verify the following inequalities
$$\frac{1}{N}\sum_{i=1}^{N}(d_{i}+\tilde{k}_{2}) \geq \left(\tilde{k}_{2}^N \hat{f}_{1}^N (d)\right)^{1/N}\geq \tilde{k}_{2} e^{\tilde{T}_R/N},$$
and consequently,
\begin{equation}
\label{feasibleTildeP2EN}
\sum_{i=1}^{N}d_{i} \geq N\tilde{k}_{2}(e^{\tilde{T}_R/N}-1).
\end{equation}
Taking into account that the lower bound of \eqref{feasibleTildeP2EN} is equal to $f_2^N(\hat{d})$, the conclusion follows.
\end{proof}

Finally, using \eqref{aproxbiglambda}, we determine the optimal solution of the approximate mixed-integer optimization problem of the palliative approach.

\begin{corollary}
\label{coroteoremaMaxP2ENbigLambda}
 Let us assume  that $d_{min}>0$ and (\ref{H2}). Then a global optimal solution of $(\hat{P}_{2})$  problem
 is given by $(\hat{N}, \hat{d})$ with
\begin{itemize}
\item [a)] $\hat{N}= N_{max}-1$  and $\hat{d}_i=d^*$,  for $i=1, \ldots, \hat{N}$, if $(N_{max}-1) d^* \leq N_{max} d_{min}$,
\item [b)] $\hat{N}= N_{max}$  and $\hat{d}_i=d_{min}$, for $i=1, \ldots, \hat{N}$, if $(N_{max}-1) d^* \geq N_{max} d_{min}$,
\end{itemize}
where  $d^*= \tilde{k}_{2}(e^{\tilde{T}_R/(N_{max}-1)}-1)$.
\end{corollary}

\begin{proof}
This can be deduced combining Theorem \ref{teoremaMaxP2ENbigLambda} with the fact that the auxiliary function
$\varphi_2(x)=x(e^{1/x}-1)$  is strictly decreasing in $(0,+\infty)$ and
\begin{equation}
f_2(N,\hat{d}) = N\tilde{k}_{2}(e^{\tilde{T}_R/N}-1) = \tilde{T}_R\tilde{k}_{2}\varphi_2\left(\frac{N}{\tilde{T}_R}\right),
\label{solf2}
\end{equation}
for $ N \leq N_{max}-1$ and using that $(\underbrace{d_{min}, \ldots,d_{min}}_{N_{max}})$ is the solution of $(\hat{P}_{2}^{N_{max}})$.
\end{proof}

\begin{remark}
\begin{itemize}
\item [i)] Let us note that both possibilities presented in Corollary \ref{coroteoremaMaxP2ENbigLambda} can occur in practical situations: for example, it can be observed in Table~\ref{2EdminTMZ}  that option $a)$ is true for $d_{min} \in\{50,100\}$, while the option $b)$ holds for $d_{min} \in \{75,150\}$.
\item [ii)] For the trivial cases we recover the expected solutions. For instance, when $L_{*} \rightarrow \theta^{-}$, we derive that $e^{\tilde{T}_R} \rightarrow 0$ and the nonlinear constraint is always verified and therefore the solution of $(\hat{P}_{2})$ is given by $(\hat{N}, \hat{d}) = (1,d_{min})$, as one can expect from the beginning.
\item [iii)] We can also observe that $N_{max}$ is a decreasing function of $d_{min}$; in fact, $N_{max} \rightarrow + \infty$ when $d_{min} \rightarrow 0$. Moreover, in case $a)$, we have
\[ d_{min} \leq d^* \leq \frac{N_{max}}{N_{max}-1}d_{min}, \]
and $d^* \approx d_{min}$ for $N_{max}$ sufficiently large. From here we deduce that case $a)$ is the prevalent one when $d_{min}$ is small enough.
\item [iv)] Once more, we can now study the dependence of the optimal cumulative dose with respect to $d_{min}$.  Taking (\ref{solf2}) into account and previous remark, we conclude that
\[f_2(\hat{N},\hat{d}) = \tilde{T}_R\tilde{k}_{2}\varphi_2\left(\frac{N_{max}-1}{\tilde{T}_R}\right) \approx \tilde{T}_R\tilde{k}_{2}\varphi_2\left(\frac{1}{\log(d_{min}/\tilde{k}_2 +1)} -\frac{1}{\tilde{T}_R}\right),\]
and to decrease $f_{2}(\hat{N},\hat{d})$ we have to take $d_{min}$ as small as possible. Again, this is in line with the principles of MC.
\end{itemize}
\end{remark}

Table~\ref{2ETMZ} and  the plot at the bottom of Figure~\ref{fig:1E} are devoted to the palliative approach with TMZ. We have considered the threshold $L_{\star}=0.1813\theta$  which is associated to the size of the tumor at the final time $T$ with the best treatment of the curative approach  (see Table~\ref{1ETMZ}). Let us remember that in this case we are considering treatments with a maximum of $40$ doses associated to the 5/28d schedule,  $d_{min}=100$  $mg /m^2$ and the parameter values of Table~\ref{ParameterValues}.  Moreover, Table~\ref{2ETMZ} presents two kind of optimal solutions. For each value of $N$ (number of doses of the treatment) the following three columns are devoted to problems $(\hat{P}_{2}^N)$ (that use \eqref{aproxbiglambda}) and the last four columns correspond to optimal solutions obtained with FMINCON for  problems $(P_2^N)$, setting the value for $N$ and without using that approximation.  Columns with labels $\bar{d}_{min}$ and $\bar{d}_{max}$ show the minimum and maximum doses of the treatment given by FMINCON, respectively. Let us note that both values in each row are almost equal. Moreover, for each value of $N$, the results for the problems $(\hat{P}_{2}^N)$ are good approximations of the results obtained with FMINCON.  Third and seventh columns (with labels $f_2^{N} $ and $(f_2^{N})_{opt} $) present the total dose associated to each optimal solution. We check that the longest treatment is the best from the point of view used in this subsection. Finally, let us note that the total dose of the approximate solution ($5{,}749.24$ $mg /m^2$) is less than the FMINCON optimal value ($5{,}749.95$ $mg /m^2$), but the constraint on the tumor volume is not satisfied (because the tumor size level is slightly exceeded). The plot at the bottom of Figure~\ref{fig:1E} shows the evolution of the tumor volume with three of these treatments computed with FMINCON.

Last but not least, in Table~\ref{2EdminTMZ} and at the bottom of Figure~\ref{fig:1Edmin} we illustrate  the influence of the level $d_{min}$. Now, we compare the optimal treatments corresponding to the palliative approach.
Again, first we have fixed $d_{min}$ and solved $(\hat{P}_{1})$; once we have determined the corresponding solution ($\hat{N}$ and ${\{\hat{d}_i\}}_{i=1}^{\hat{N}}$) we have selected the more appropriated schedules used for TMZ. All the treatments in the table keep the tumor below the given $L_{\star}$ level, and the longest treatment (that also has the lowest doses) is the best because it has the lowest cumulative dose. These results are also consistent with MC. The corresponding tumor volumes, $L(t)/\theta$ , for a medium-sized initial tumor  ($25\%$ of carrying capacity) with the treatments in Table~\ref{2EdminTMZ} are shown in Figure~\ref{fig:1Edmin}.

The computations were performed on a 3.9 GHz Core i5-8265U machine, with 8 GB RAM running under the 64-bit version of Windows 10 and MATLAB R2019b.

\begin{table}
\caption{Results for $(P_{2})$ and TMZ with $L_{\star}= 0.1813\theta$ and $d_{min} = 100$ $mg/m^2$ with schedule 5/28d.}
\label{2ETMZ}
\begin{tabular}{l|ccc||cccc}
\hline\noalign{\smallskip}
 & \multicolumn{3}{|c||}{$(\hat{P}_{2})$} &  \multicolumn{4}{|c}{$(P_{2})$} \\
   \noalign{\smallskip}
\hline\noalign{\smallskip}
    $N$  &  $\hat{d}_i$ & $f_2^{N} $  & $L(T)/\theta$ &  $\bar{d}_{min}$ & $\bar{d}_{max}$ & $(f_2^{N})_{opt} $ & $L(T)/\theta$  \\
\noalign{\smallskip}\hline\noalign{\smallskip}
  $  33 $ & $ 196.18 $ & $ 6{,}473.84 $ & $ 0.18134 $ & $ 196.20957 $ & $ 196.21000 $ & $ 6{,}474.92  $ & $ 0.18130 $
 \\
  $  34 $ & $ 186.60 $ & $ 6{,}344.54 $ & $ 0.18134 $ & $ 186.63372 $ & $ 186.63413  $ & $ 6{,}345.55  $ & $ 0.18130 $
   \\
   $  35 $ & $ 177.87 $ & $ 6{,}225.51 $ & $ 0.18134 $ & $ 177.89874 $ & $ 177.89913 $ & $ 6{,}226.46$ & $ 0.18130 $
 \\
   $  36 $ & $ 169.88 $ & $ 6{,}115.59 $ & $ 0.18134 $ & $ 169.90161 $ & $ 169.90211 $ & $ 6{,}116.46$ & $ 0.18130 $
\\
  $  37 $ & $ 162.54 $ & $ 6{,}013.81 $ & $ 0.18134 $ & $ 162.55750 $ & $ 162.55786  $ & $  6{,}014.63 $ & $ 0.18130$
 \\
  $  38 $ & $ 155.77 $ & $ 5{,}919.30 $ & $ 0.18134 $ & $ 155.79148 $ & $ 155.79182  $ & $ 5{,}920.08$ & $ 0.18130 $
 \\
   $  39 $ & $ 149.52 $ & $ 5{,}831.32 $ & $ 0.18133 $ & $ 14
   9.54009 $ & $ 149.54041  $ & $ 5{,}832.07$ & $ 0.18130 $
 \\
   $  \mathbf{40} $ & $ 143.73 $ & $ 5{,}749.24$ & $ 0.18133 $ & $ \mathbf{143.74862 } $ & $ \mathbf{ 143.74893} $ & $ \mathbf{5{,}749.95} $ & $ 0.18130 $
\\
\noalign{\smallskip}\hline
\end{tabular}
\end{table}

\begin{table}
\caption{Numerical results for $(\hat{P}_{2})$ with   $L_{\star}= 0.1813\theta$ varying $d_{min}$.}
\label{2EdminTMZ}
\begin{tabular}{r|ccc|rcc}
\hline\noalign{\smallskip}
$d_{min}$  &  $N$  &  $\hat{d}_i$ & $f_2^{N} $ & Schedule & $L(T)/\theta$ &  Dose intensity  \\
\noalign{\smallskip}\hline\noalign{\smallskip}
$150$  &  $ 39 $ & $ 150.00 $ & $ 5{,}850.00 $ &  7/14d &  $   0.18047 $ & $ 79.05 $  \\ 
$100$  &  $ 51 $ & $ 100.25 $ & $ 5{,}112.64 $ & 21/28d &  $   0.18133 $ & $ 78.66 $ \\
$75 $  &  $ 63 $ & $  75.00 $ & $ 4{,}725.00 $ & 21/28d &  $   0.18119 $ & $ 61.36 $ \\
$50 $  &  $ 86 $ & $  50.29 $ & $ 4{,}324.78 $ & 28/28d &  $   0.18132 $ & $ 50.29 $ \\
\noalign{\smallskip}\hline
\end{tabular}
\end{table}

\section{Conclusions}

Throughout this work we have given a mathematical justification supporting metronomic chemotherapy (MC) as the best option for most cytotoxic drugs (i.e. when the main hypothesis \eqref{MH} holds) and for curative and palliative approaches in oncology. First, we have obtained  explicit expressions  of global optimal solutions of $(\hat{P}_1)$ and $(\hat{P}_2)$ problems (given by Corollary~\ref{coroteoremaMaxP1EN} and Corollary~\ref{coroteoremaMaxP2ENbigLambda}) that are approximations of $(P_1)$ and $(P_2)$, respectively. Once more, we emphasize that the administration times  do not appear in the formulation of $(\hat{P}_1)$ and $(\hat{P}_2)$ and for both cases the optimal solutions lead to the longest feasible treatments with equal doses, which are in line with MC. We further deduce mathematically that it is convenient to take the effectiveness level $d_{min}$ as small as possible, but above MED  (i.e. ensuring a minimum effect!). Numerical experiments with TMZ support the high quality of the solutions of $(\hat{P}_1)$ and $(\hat{P}_2)$ with respect to those of $(P_1)$ and $(P_2)$, respectively, and  illustrate the  theoretical results. Moreover, the optimal treatments calculated for both curative and palliative approaches are very similar. Finally, with our method the administration times are adjusted a posteriori (once the number of doses and the individual doses of the treatment are determined).

In our work we have used a simple mathematical model being crucial  the Norton-Simon hypothesis, the Emax model, the PK of the drug  given by \eqref{pCauchy} and the main hypothesis \eqref{MH}.
In the future, we would like to analyze several generalizations and extensions  from a mathematical point of view, including some relevant factors such as drug resistance and anti-angiogenesis effect or the case when  \eqref{MH} fails. We think that both theoretical and numerical results would be very useful for clinical purposes.

\vspace{1cm}

\it{Acknowledgements.
The first and second authors were supported by MCIN/ AEI/10.13039/501100011033/ under research
projects MTM2017-83185-P and PID2020-114837GB-I00.}

\end{document}